\documentclass{amsart}
\usepackage[latin1]{inputenc}
\usepackage[english]{babel}
\usepackage{amsmath}
\usepackage{amsfonts}
\usepackage{amssymb}
\usepackage{epsfig}
\usepackage{amsopn}
\usepackage{amsthm}
\usepackage{color}
\usepackage{graphicx}
\usepackage{subfigure}
\usepackage{enumerate}
\newtheorem{theorem}{Theorem}[section]

\newtheorem{lemma}[theorem]{Lemma}
\newtheorem{proposition}[theorem]{Proposition}

\newtheorem*{theorem*}{Theorem}
\newtheorem*{lemma*}{Lemma}
\newtheorem*{remark*}{Remark}
\newtheorem*{definition*}{Definition}
\newtheorem*{proposition*}{Proposition}
\newtheorem*{corollary*}{Corollary}
\numberwithin{equation}{section}
\newcommand{\real}{\mathbb{R}}

\let\ced=\c         % cedilla
\def\qed{\,\unskip\kern 6pt \penalty 500
\raise -2pt\hbox{\vrule \vbox to8pt{\hrule width 6pt
\vfill\hrule}\vrule}\par}
\definecolor{darkblue}{rgb}{0.05, .05, .65}
\definecolor{darkgreen}{rgb}{0.1, .65, .1}
\definecolor{darkred}{rgb}{0.8,0,0}
\newcommand{\beqn}{\begin{equation}}
\newcommand{\eeqn}{\end{equation}}
\newcommand{\bear}{\begin{eqnarray}}
\newcommand{\eear}{\end{eqnarray}}
\newcommand{\bean}{\begin{eqnarray*}}
\newcommand{\eean}{\end{eqnarray*}}
\begin{document}
%%%%%%%%%%%%%%%%%%%%%%%%%%%%%%%%%%%%%%%%%%%%%%%%%

%%%%%%%%%%%%%%%%%%%%%%%%%%%%%%%%%%%%%%%%%%%%%%%%%
\title[Global self-similar solutions for Hardy-H\'enon equations]{Global self-similar solutions for Hardy-H\'enon equations with linear and quasilinear diffusion}

\author[R. G. Iagar]{Razvan Gabriel Iagar}
\address[R. G. Iagar]{Departamento de Matem\'atica Aplicada,
	Ciencia e Ingenieria de los Materiales y Tecnologia Electr\'onica,
	Universidad Rey Juan Carlos, M\'ostoles, 28933, Madrid, Spain}
\email{razvan.iagar@urjc.es}

\author[A. Sánchez]{Ariel S\'anchez}
\address[A. Sánchez]{Departamento de Matem\'atica Aplicada,
	Ciencia e Ingenieria de los Materiales y Tecnologia Electr\'onica,
	Universidad Rey Juan Carlos, M\'ostoles, 28933, Madrid, Spain}
\email{ariel.sanchez@urjc.es}

\author[E. Sarrión-Pedralva]{Erik Sarri\'on-Pedralva}
\address[E. Sarrión-Pedralva]{Departamento de Matem\'atica Aplicada,
	Ciencia e Ingenieria de los Materiales y Tecnologia Electr\'onica,
	Universidad Rey Juan Carlos, M\'ostoles, 28933, Madrid, Spain}
\email{erik.sarrion@urjc.es}
\date{}

\maketitle

\begin{abstract}
Global self-similar solutions to the parabolic Hardy-H\'enon equation
$$
u_t=\Delta u^m+|x|^{\sigma}u^p, \quad (x,t)\in\real^N\times(0,\infty),
$$
are classified in the range of exponents $m\geq1$, $p>m$ and $\sigma>\max\{-2,-N\}$. The classification varies strongly with respect to the celebrated \emph{Fujita} and \emph{Sobolev critical exponents}
$$
p_F(\sigma)=m+\frac{\sigma+2}{N}, \quad p_S(\sigma)=
	\begin{cases}
        \frac{m(N+2\sigma+2)}{N-2}, & \mbox{if } N\geq3, \\[1mm]
    	   \infty, & \mbox{if } N\in\{1,2\}.
    \end{cases}
$$
Indeed, if $p\in(p_F(\sigma),p_S(\sigma))$, both equations admit self-similar solutions with either compact support (if $m>1$) or Gaussian-like tail as $|x|\to\infty$ (if $m=1$), as well as a one-parameter family satisfying
$$
u(x,t)\sim C|x|^{-(\sigma+2)/(p-m)}, \quad {\rm as} \ |x|\to\infty.
$$
If $p\geq p_S(\sigma)$, there are only self-similar solutions with the latter algebraic tail, while for $m<p\leq p_F(\sigma)$ no global solutions exist. The results open the way for a deeper study of the role of these solutions in the dynamics of the Hardy-H\'enon equations.
\end{abstract}

\medskip


\medskip


\section{Introduction}

In the last decade, an increasing interest on the study of qualitative properties and large time behavior of solutions to a class of reaction-diffusion equations involving a spatially nonhomogeneous term has been observed. This class of equations is usually referred in literature as \emph{parabolic Hardy-H\'enon equations}, the name stemming from Henon's work \cite{He73} where the elliptic counterpart of the equations has been deduced as a model in astrophysics, and Baras and Goldstein's paper \cite{BG84} where existence and non-existence of solutions to the critically singular equation of the family is strongly related to the optimal constant of the Hardy's inequality. To be more specific, the standard parabolic Hardy-H\'enon equation reads
\begin{equation}\label{eq0}
u_t=\Delta u+|x|^{\sigma}u^p, \quad (x,t)\in\real^N\times(0,T),
\end{equation}
with $p>1$ and $\sigma>\max\{-2,-N\}$, where $N$ designs the dimension of the space we are working in. The main feature of it is the competition between the diffusion and the reaction terms, the latter being weighted by an unbounded variable coefficient. Let us note here that, if $\sigma>0$, we deal with a weight which is unbounded at infinity, while if $\sigma<0$, this weight becomes singular at $x=0$. This latter case is usually known as the Hardy equation and, from the functional analytic point of view, its study is more involved.

Precisely the complexity of the functional-analytic properties of Eq. \eqref{eq0}, related on the one hand to the formulation of a sharp theory of well-posedness depending on the regularity of the initial condition and, on the other hand, to the establishment of conditions on the solutions for either finite time blow-up or global existence (the latter meaning $u(t)$ bounded for any $t\in(0,\infty)$), has attracted the attention of a number of authors in recent years, see for example \cite{BS19, BSTW17, CIT21, CIT22, CITT24, MS21, SU24}, to name but a few of the works dealing with the analysis of Eq. \eqref{eq0}. In these works, the influence of the weight $|x|^{\sigma}$ has been revealed. Indeed, the optimal functional spaces ensuring existence and uniqueness of solutions strongly depend on $\sigma$, as well as the asymptotic behavior of a suitable class of solutions (see for example \cite{BSTW17}). Moreover, the coefficient $|x|^{\sigma}$ has a strong influence also on the properties related to finite time blow-up, as shown in \cite{MS21, SU24}.

A nontrivial generalization of the parabolic Hardy-H\'enon equation \eqref{eq0} is the following equation involving a quasilinear diffusion
\begin{equation}\label{eq1}
u_t=\Delta u^m+|x|^{\sigma}u^p, \quad (x,t)\in\real^N\times(0,\infty),
\end{equation}
with $m>1$ and the same conditions as for Eq. \eqref{eq0} on the exponents $p$ and $\sigma$. An important starting point in the qualitative theory of Eq. \eqref{eq1} is the paper \cite{AdB91}. Despite the fact that it deals with the weight $(1+|x|)^{\sigma}$ instead of the pure power $|x|^{\sigma}$, the local existence and uniqueness results therein remain valid also for Eq. \eqref{eq1} at least for $\sigma>0$. Related to the alternative between blow-up and global existence, the \emph{Fujita-type exponent} to Eq. \eqref{eq1}, that is
\begin{equation}\label{pF}
p_F(\sigma)=m+\frac{2+\sigma}{N},
\end{equation}
has been derived in \cite{Qi98}, where a global solution for $p>p_F(\sigma)$ is constructed as well. Let us recall here that, in analogy to the seminal paper by Fujita \cite{Fu66}, the exponent $p_F(\sigma)$ has the property that, for $m<p\leq p_F(\sigma)$, any nontrivial solution presents finite time blow-up, while for $p>p_F(\sigma)$ both global solutions and solutions with finite time blow-up exist, and finding conditions to classify the solutions with respect to this alternative is a very interesting problem, analyzed for example in \cite{Su02}. An analysis of the behavior of solutions near the blow-up time is performed in \cite{AT05}.

Let us finally mention that a sharp theory of existence and uniqueness of solutions to Eq. \eqref{eq1} in the range $\sigma>0$, together with a number of other properties related to the evolution of interfaces, has been achieved by two of the authors in \cite{ILS24}. Moreover, an optimal theory of existence, uniqueness and regularity of solutions to Eq. \eqref{eq1} in the singular case $\sigma\in(\max\{-2,-N\},0)$, together with strong conditions on the initial data leading to either finite time blow-up or global existence, is performed in the recent work \cite{IL26}.

\medskip

\noindent \textbf{Self-similar solutions.} It is a well-established fact that nonlinear diffusion equations such as \eqref{eq0} and \eqref{eq1} are described by a class of solutions enjoying a specific form, known as \emph{self-similar solutions}. Such solutions are, in many cases, prototypes for the typical expected behavior of general solutions of an equation (in the sense of global behavior or finite time blow-up) and, at the same time, patterns for the large time behavior of rather general classes of solutions. The interested reader can find more about the role of self-similarity in the making of the theory of nonlinear diffusion equations in the monographs by V\'azquez \cite{VPME, VazSmooth}.

The present work is devoted to the classification of global in time radially symmetric self-similar solutions to the parabolic Hardy-H\'enon equations \eqref{eq0} and \eqref{eq1} in the range of exponents (written in an unified form which includes both $m=1$ and $m>1$)
\begin{equation}\label{range.exp}
m\geq1, \quad p>p_F(\sigma), \quad \sigma\in(\max\{-2,-N\},\infty),
\end{equation}
mentioning at this point that, in the complementary range $1<p\leq p_F(\sigma)$, all the non-trivial solutions to Eq. \eqref{eq1} blow up in finite time according to \cite{Pi97, Qi98, ILS24, IL26}. More precisely, we look for solutions of the form
\begin{equation}\label{forward.SS}
u(x,t)=t^{-\alpha}f(\xi), \quad \xi=|x|t^{-\beta},
\end{equation}
for suitable exponents $\alpha$ and $\beta$. Inserting the previous form in Eq. \eqref{eq1} and equating the time variable in the resulting terms, we deduce by straightforward calculations that
\begin{equation}\label{SSexp}
\alpha=\frac{\sigma+2}{\sigma(m-1)+2(p-1)}, \quad \beta=\frac{p-m}{\sigma(m-1)+2(p-1)},
\end{equation}
and observe that the exponents in \eqref{SSexp} are both positive in the range \eqref{range.exp}. It then follows that the profile $f$ solves the following non-autonomous differential equation
\begin{equation}\label{ODE.forward}
(f^m)''(\xi)+\frac{N-1}{\xi}(f^m)'(\xi)+\alpha f(\xi)+\beta\xi f'(\xi)+\xi^{\sigma}f^p(\xi)=0.
\end{equation}

Eq. \eqref{ODE.forward} has been thoroughly studied for $\sigma=0$ in the semilinear case $m=1$, where it is the radially symmetric reduction of the differential equation sometimes referred as \emph{the Haraux-Weissler equation}, proposed and analyzed in \cite{HW82, W86} as an interesting case of non-uniqueness in suitable functional spaces. Both rapidly decaying (that is, with a Gaussian-like tail as $\xi\to\infty$) and slowly decaying (that is, with a power-like tail as $\xi\to\infty$) self-similar profiles $f$ are identified, and the uniqueness of the rapidly decaying, Gaussian-like profile is established in two steps in \cite{Y96, DH98}. Later on, several papers by Naito focused on the slowly decaying profiles.  First, a classification of them according to their precise behavior at infinity is given in \cite{N06, N08}. Later, their role of threshold between global solutions and solutions presenting finite time blow-up when $p>p_F(0)$, as well as asymptotic pattern for a more general class of global solutions under the only condition of a specific behavior as $|x|\to\infty$, is analyzed in \cite{N12, N20} (see also references therein). Some global self-similar solutions to Eq. \eqref{eq0} are also considered in \cite{BSTW17} as large time behavior patterns for more general solutions.

Let us only mention here that Eq. \eqref{eq0} also admits self-similar solutions presenting a finite time blow-up (which are outside the scope of this work). The ranges of existence of such solutions are now well-known. A first classification given in \cite{FT00} for $p$ below a critical exponent known as the Joseph-Lundgren exponent has been completed recently by two of the authors who, in \cite{IMS26}, analyzed the complementary range of $p$ larger than the Joseph-Lundgren exponent and revealed a very strong and unexpected influence of the exponent $\sigma$.

For $m>1$ and $\sigma=0$, a study of Eq. \eqref{ODE.forward} leading to a classification of global self-similar solutions is performed in \cite[Theorem 4, Chapter IV.3.4]{S4}, where both compactly supported profiles and profiles with an algebraic decay as $\xi\to\infty$ are identified, depending on the range of the exponents $p$. To the best of our knowledge, the uniqueness of the compactly supported profile is not yet proved. More recently, two of the authors started a systematic study of equations in the form of Eq. \eqref{eq1} and their self-similar solutions (either with finite time blow-up or global in time) are analyzed and classified in a number of papers, among which we quote \cite{ILS23, IMS23, IS22, IS25b} (see also references therein). This analysis has been extended also to the fast diffusion range $m<1$, where global self-similar solutions have been identified and analyzed in \cite{IMS23b, IS26}. Despite the fact that the range $m<1$ is out of the scope of this work, the techniques that we employ throughout this paper follow those given in these works.

We next introduce and explain our results.

\medskip

\noindent \textbf{Main results.} As outlined above, our objective is to determine the optimal ranges of existence and non-existence of global  self-similar solutions in the form \eqref{forward.SS} to Eqs. \eqref{eq1} and \eqref{eq0} in the range of exponents \eqref{range.exp}. The forthcoming classification is based on their behavior as $|x|\to\infty$. This behavior depends on two important critical exponents: the first of them is the Fujita-type exponent $p_F(\sigma)$ defined in \eqref{pF}, and the second one is the Sobolev critical exponent $p_S(\sigma)$ defined below. We also introduce, in the same definition, a third critical exponent $p_c(\sigma)$ which will play a rather technical role in some of the proofs.
\begin{equation}\label{pSpc}
p_c(\sigma):=\left\{\begin{array}{ll}\frac{m(N+\sigma)}{N-2}, & N\geq3,\\ \infty, & N\in\{1,2\},\end{array}\right.
\qquad p_S(\sigma):=\left\{\begin{array}{ll}\frac{m(N+2\sigma+2)}{N-2}, & N\geq3,\\ \infty, & N\in\{1,2\}.\end{array}\right.
\end{equation}
Let us first observe that these critical exponents are ordered: indeed, in the range of exponents \eqref{range.exp}, we have
$$
p_S(\sigma)-p_c(\sigma)=\frac{m(\sigma+2)}{N-2}>0, \quad p_c(\sigma)-p_F(\sigma)=\frac{(\sigma+2)(mN-N+2)}{N(N-2)}>0,
$$
provided $N\geq3$.

We thus focus on the profiles $f$ solutions to \eqref{ODE.forward} satisfying the initial conditions
\begin{equation}\label{icpos}
f(0)=A\in(0,\infty), \quad f'(0)=0, \quad {\rm if} \ \sigma>0,
\end{equation}
or $f(0)=A\in(0,\infty)$ and
\begin{equation}\label{icneg}
f(\xi)=\left[A^{m-p}+\frac{p-m}{m(N+\sigma)(\sigma+2)}\xi^{\sigma+2}\right]^{-1/(p-m)}+o\left(\xi^{\sigma+2}\right),
\quad {\rm if} \ \sigma\in(-2,0),
\end{equation}
as $\xi\to\infty$. As we shall see from the analysis, the conditions \eqref{icpos} and \eqref{icneg} ensure the uniqueness of a profile $f$ solving \eqref{ODE.forward} and, for simplicity, we will denote throughout the paper by $f(\cdot;A)$ this unique profile satisfying the previous initial conditions.

With this notation, we can now state our main results, starting from the \textbf{semilinear case} $m=1$.
\begin{theorem}\label{th.global.heat}
Let $m=1$ and $p$, $\sigma$ be as in \eqref{range.exp}. Then
\begin{enumerate}
  \item If $p_F(\sigma)<p<p_S(\sigma)$, there exist $A_*<A^*\in(0,\infty)$ such that the profile $f(\cdot;A^*)$ has an exponential decay as $\xi\to\infty$, more precisely,
\begin{equation}\label{beh.Q5heat}
\lim\limits_{\xi\to\infty}e^{\xi^2/4}f(\xi;A^*)=K\in(0,\infty),
\end{equation}
and, for any $A\in(0,A_*)$, the profile $f(\cdot;A)$ has an algebraic decay as $\xi\to\infty$; that is,
\begin{equation}\label{beh.Q1heat}
\lim\limits_{\xi\to\infty}\xi^{(\sigma+2)/(p-1)}f(\xi;A)=L(A)\in(0,\infty).
\end{equation}
  \item If $N\geq3$ and $p\geq p_S(\sigma)$, then all the profiles $f(\cdot;A)$ behave as in \eqref{beh.Q1heat}, for any $A\in(0,\infty)$.
\end{enumerate}
\end{theorem}
We thus observe that at least one special profile, with a Gaussian-like behavior as $\xi\to\infty$ similar to the one of the pure heat equation, exists only in a limited range of $p$, corresponding to the second item in Theorem \ref{th.global.heat}. Another interesting fact to be noticed at this point is that \emph{all} the profiles $f(\cdot;A)$ generate positive self-similar solutions to Eq. \eqref{eq0} for $p\geq p_S(\sigma)$.

In the \textbf{quasilinear range} $m>1$, the outcome of our analysis generalizes the statement of Theorem \ref{th.global.heat}, with a significant change stemming from the properties of the diffusion term.
\begin{theorem}\label{th.global.super}
Assume $m>1$ and $p$, $\sigma$ as in \eqref{range.exp}. Then
\begin{enumerate}
\item If $p_F(\sigma)<p<p_S(\sigma)$, there exists $A^*\in(0,\infty)$ such that the profile $f(\cdot;A^*)$ is compactly supported, that is, there exists $\xi_0\in(0,\infty)$ such that
\begin{equation}\label{beh.Q5f}
f(\xi_0;A)=0, \quad f(\xi;A)>0 \ {\rm for} \ \xi\in(0,\xi_0), \quad (f^m)'(\xi_0;A)=0.
\end{equation}
Moreover, there is $A_*\in(0,A^*)$ such that, for any $A\in(0,A_*)$, the profile $f(\cdot;A)$ has an algebraic decay as $\xi\to\infty$ as follows
\begin{equation}\label{beh.Q1f}
\lim\limits_{\xi\to\infty}\xi^{(\sigma+2)/(p-m)}f(\xi;A)=L(A)\in(0,\infty).
\end{equation}
\item If $N\geq 3$ and $p\geq p_S(\sigma)$, then all the profiles $f(\cdot;A)$ behave as in \eqref{beh.Q1f}, for any $A\in(0,\infty)$.
\end{enumerate}
\end{theorem}

A number of self-similar profiles $f(\cdot;A)$, corresponding to both ranges of the exponent $p$ considered in Theorem \ref{th.global.super}, are plotted in Figures \ref{fig1} and \ref{fig2}. 

In view of the uniqueness result for $\sigma=0$ established in \cite{Y96} (for $p\in(m,p_c(0))$) and in \cite{DH98} (for $p\in(p_c(0),p_S(0))$, one might expect that the value $A^*$ corresponding to the second item in Theorems \ref{th.global.heat} and \ref{th.global.super} is unique; that is, the self-similar solution to Eq. \eqref{eq0} with a Gaussian-like behavior as $|x|\to\infty$ and the compactly supported self-similar solution to Eq. \eqref{eq1} are also unique for $\sigma\neq0$. However, this uniqueness result overpasses by far the scope of the present work. At first, this uniqueness is still missing from literature for $m>1$ even when $\sigma=0$; as for $m=1$, one can check in the above mentioned references that the proof of the uniqueness, even for $\sigma=0$, is technically extremely involved. We will devote a forthcoming work to this uniqueness result for $m=1$.

\medskip

\noindent \textbf{Remark. Regularity of the self-similar solutions}. All the self-similar profiles described in Theorem \ref{th.global.heat} remain strictly positive for any $\xi>0$, which immediately implies that the associated solutions belong to the class $C^2(\real^N\setminus\{0\})$. Among the self-similar profiles described in Theorem \ref{th.global.super}, the compactly supported ones are no longer classical at the limit of their support, and in fact they are only Holder continuous. However, the condition on the derivative in \eqref{beh.Q5f} ensures that the contact with zero at $\xi=\xi_0\in(0,\infty)$ is sufficiently smooth in order to satisfy the weak formulation, as explained in \cite[Section 9.8]{VPME}.

The other potential lack of regularity may occur at $\xi=0$. A closer analysis of the local behavior up to second order shows that the profiles $f(\cdot;A)$ satisfy the following local expansion as $\xi\to0$
\begin{equation}\label{beh.P0f}
f(\xi;A)\sim\left\{\begin{array}{ll}\left[A^{m-1}-\frac{\alpha(m-1)}{2mN}\xi^2\right]^{1/(m-1)}, & \sigma>0, \ m>1,\\[1.5mm]
A\exp\left(-\frac{\alpha}{2N}\xi^2\right), & \sigma>0, \ m=1,\\[1mm]
\left[A^{m-p}+\frac{p-m}{m(N+\sigma)(\sigma+2)}\xi^{\sigma+2}\right]^{-1/(p-m)}, & \sigma\in(\max\{-2,-N\},0), \ m\geq1.
\end{array}\right.
\end{equation}
In view of this local expansion, the self-similar solutions in the form \eqref{forward.SS} with profiles $f(\cdot;A)$:

$\bullet$ are of class $C^2$ also at $x=0$ when $\sigma>0$. Therefore, they are solutions in the classical sense, with the exception of the compactly supported ones given in Theorem \ref{th.global.super} for $m>1$.

$\bullet$ belong to $C^1(\real^N)$ but not to $C^2(\real^N)$, for $\sigma\in[-1,0)$. In this case, the corresponding self-similar solutions are weak solutions to Eq. \eqref{eq1}, but not classical ones.

$\bullet$ are no longer of class $C^1$ at $\xi=0$ and exhibit a peak at the origin for $N\geq2$ and $\sigma\in(-2,-1)$. For these self-similar solutions $u$, the singularity of $\nabla u^m$ at $x=0$ expands as $|x|^{\sigma+1}$, ensuring that $\nabla u^m\in L^2(\real^N)$. Consequently, they remain weak solutions to Eq. \eqref{eq1}. Similar phenomena have been observed in problems involving singular weights, see \cite{RV06}, where such solutions with peaks at the origin are essential in the description of the large time behavior. A fully precise characterization of the optimal regularity at $x=0$ of these self-similar solutions is deduced following the approach in \cite[Section 3.3]{IL25}.

\section{A dynamical system. Phase space analysis}\label{sec.local}

Due to its generality, throughout this section we fix $m > 1$ and, in order to simplify the discussion, we focus first on dimensions $N\geq3$, postponing dimensions $N=1$ and $N=2$ for discussion at the end of it. The changes required in the semilinear case $m=1$ will be considered in a separate subsection. Furthermore, we introduce the following change of variables, previously employed in \cite{IS25b, IS26}:
\begin{equation}\label{PSchange}
X(\eta)=\frac{\alpha}{m}\xi^2f^{1-m}(\xi), \qquad Y(\eta)=\frac{\xi f'(\xi)}{f(\xi)}, \qquad Z(\eta)=\frac{1}{m}\xi^{\sigma+2}f^{p-m}(\xi),
\end{equation}
where $\eta=\ln\,\xi$ denotes a new independent variable and $\alpha$ is given in \eqref{SSexp}. After straightforward computations, Eq. \eqref{ODE.forward} transforms into the quadratic dynamical system
\begin{equation}\label{PSsyst}
\left\{\begin{array}{ll}\dot{X}=X(2-(m-1)Y), \\ \dot{Y}=-X-(N-2)Y-Z-mY^2-\frac{p-m}{\sigma+2}XY, \\ \dot{Z}=Z(\sigma+2+(p-m)Y).\end{array}\right.
\end{equation}
We proceed by analyzing the trajectories in the phase space associated with system \eqref{PSsyst}, extracting from them information about the profiles and their asymptotic behavior. This requires a detailed study of the critical points of the system, both finite and infinite. In some parts, we will be brief, as many technical arguments closely resemble those already presented in detail in previous works by the authors such as  \cite{ILS23,IMS23b,IS26}. Note first that, within our range of interest (non-negative self-similar profiles), we have $X\geq0$, $Z\geq0$, and that the planes $\{X=0\}$ and $\{Z=0\}$ remain invariant under the dynamics of the system \eqref{PSsyst}.

\subsection{Finite critical points}

A simple inspection of the system \eqref{PSsyst} shows that there are three finite critical points in the range $m>1$, namely
\begin{equation*}
P_0=(0,0,0), \ \ P_1=\left(0,-\frac{N-2}{m},0\right), \ \ P_2=\left(0,-\frac{\sigma+2}{p-m},\frac{(N-2)(\sigma+2)(p-p_c(\sigma))}{(p-m)^2}\right).
\end{equation*}
We next present the local analysis of these critical points, noting that the point $P_2$ exists only when $p>p_c(\sigma)$ and $P_2=P_1$ for $p=p_c(\sigma)$ (recalling that $p_c(\sigma)$ is defined in \eqref{pSpc}). The behavior of the orbits in a neighborhood of these points is summarized in the following technical result, whose proof is completely identical to that of \cite[Lemmas 2.1-2.3]{IS26}.
\begin{lemma}\label{lem.P0P1P2}
Consider the dynamical system \eqref{PSsyst} and its critical points $P_0$, $P_1$, and $P_2$. Then their local structure and the behavior of the associated profiles as $\xi\to0$ are as follows:
\begin{enumerate}
\item in a neighborhood of $P_0$, the system admits a two-dimensional unstable manifold and a one-dimensional stable manifold. The unstable manifold can be described as the one-parameter family of trajectories
\begin{equation}\label{lC}
(l_C): \ \left\{\begin{array}{ll}Y(\eta)=-\frac{X(\eta)}{N}-\frac{Z(\eta)}{N+\sigma}+o(|X(\eta),Z(\eta)|), \\[1mm]
Z(\eta)\sim CX(\eta)^{(\sigma+2)/2},\end{array}\right. \quad {\rm as} \ \eta\to-\infty,
\end{equation}
where $C\in[0,\infty]$ (understanding $C=0$ for the trajectory contained in the invariant plane $\{Z=0\}$ and $C=\infty$ for the trajectory contained in the invariant plane $\{X=0\}$). Trajectories leaving this point generate profiles with the local behavior described in \eqref{beh.P0f} and, in particular, the trajectory $(l_C)$ for a given $C\in(0,\infty)$ corresponds to the profile $f(\cdot;A)$ with
\begin{equation}\label{bij}
A=(Cm)^{2/L}\left(\frac{\alpha}{m}\right)^{(\sigma+2)/L}, \quad L:=\sigma(m-1)+2(p-1).
\end{equation}
\item The critical point $P_1$ is either an unstable node if $p<p_c(\sigma)$ or a saddle point with a two-dimensional unstable manifold contained in the invariant plane $\{Z=0\}$ and a one-dimensional stable manifold contained in the invariant plane $\{X=0\}$ if $p\geq p_c(\sigma)$. In the range $p<p_c(\sigma)$, the trajectories leaving $P_1$ correspond to profiles presenting a vertical asymptote of the form
\begin{equation}\label{beh.P1.super}
f(\xi)\sim C\xi^{-(N-2)/m}, \qquad {\rm as} \ \xi\to0, \qquad C>0.
\end{equation}
\item Let $p>p_c(\sigma)$. Then the critical point $P_2$ is either an unstable node or focus if $p\in(p_c(\sigma),p_S(\sigma))$ or a saddle point with a stable two-dimensional manifold contained in the invariant plane $\{X=0\}$ and a unique orbit going out of it if $p\geq p_S(\sigma)$. In both cases, the trajectories contained in the unstable manifold of the point $P_2$ correspond to profiles presenting a vertical asymptote in the form
\begin{equation}\label{beh.P2.super}
f(\xi)\sim \left[\frac{m(\sigma+2)(N-2)(p-p_c(\sigma))}{(p-m)^2}\right]^{1/(p-m)}\xi^{-(\sigma+2)/(p-m)}, \ {\rm as} \ \xi\to0.
\end{equation}
\end{enumerate}
\end{lemma}
\begin{proof}[Proof of the statement related to $P_0$]
As mentioned before, the proof is identical to the one of \cite[Lemmas 2.1-2.3]{IS26} and, in fact, follows almost straightforwardly from the calculation of the Jacobian matrix of the vector field of the system \eqref{PSsyst} at the critical points, as all of them are hyperbolic. For the sake of completeness, we give here a sketch of the description of the unstable manifold of $P_0$, as it will be useful in the sequel. The linearization of the system \eqref{PSsyst} at $P_0$ has the matrix
$$
\left(
  \begin{array}{ccc}
    2 & 0 & 0 \\
    -1 & -(N-2) & -1 \\
    0 & 0 & \sigma+2 \\
  \end{array}
\right),
$$
with eigenvalues $\lambda_1=2$, $\lambda_2=-(N-2)<0$, $\lambda_3=\sigma+2>0$ and corresponding eigenvectors $e_1=(N,-1,0)$, $e_2=(0,1,0)$ and $e_3=(0,1,-(N+\sigma))$. We then deduce from the Stable Manifold Theorem \cite[Section 2.7]{Pe} that the unstable manifold of $P_0$ is tangent to the subspace spanned by the eigenvectors $e_1$ and $e_3$, which readily gives the linear approximation of $Y(\eta)$ in \eqref{lC}. Moreover, we infer from the first and third equation of \eqref{PSsyst} that, in a first approximation, we have
\begin{equation}\label{interm3}
Z(\eta)\sim CX(\eta)^{(\sigma+2)/2}, \qquad {\rm as} \ \eta\to-\infty,
\end{equation}
for any $C\in[0,\infty)$, completing the proof of the approximation $l_C$. Note here that the limiting case $C=0$ in \eqref{interm3} corresponds directly to the trajectory fully contained in the invariant plane $\{Z=0\}$, while for the limiting case $C\to\infty$, we can write \eqref{interm3} equivalently as
$$
X(\eta)\sim\left[\frac{1}{C}Z(\eta)\right]^{2/(\sigma+2)}, \qquad {\rm as} \ \eta\to-\infty,
$$
and $C=\infty$ clearly corresponds in this way to the unique trajectory leaving $P_0$ contained in the invariant plane $\{X=0\}$. The correspondence \eqref{bij} follows by direct calculation from \eqref{interm3} after replacing $X$, $Z$ by their definitions in \eqref{PSchange}. Noting that
\begin{equation}\label{interm2}
\frac{Z(\eta)}{X(\eta)}\sim CX(\eta)^{\sigma/2}\to\begin{cases}
                                                    0, & \mbox{if } \sigma>0 \\
                                                    \infty, & \mbox{if } \sigma<0
                                                  \end{cases}, \qquad {\rm as} \ \eta\to-\infty,
\end{equation}
the local expansion \eqref{beh.P0f} is readily deduced by an integration step after neglecting in each case the lower order term between $X(\eta)$ and $Z(\eta)$ in \eqref{lC} and then undoing the change of variable \eqref{PSchange}. Details of this step are easy and omitted here, since they are given in the proof of \cite[Lemma 2.1]{IS26}.
\end{proof}

\noindent \textbf{Remark.} When $p>p_S(\sigma)$, the solution corresponding to the only orbit contained in the unstable manifold of $P_2$ can be explicitly written, as it is the stationary solution
\begin{equation*}
u(x)=K|x|^{-(\sigma+2)/(p-m)}, \qquad K=\left[\frac{m(\sigma+2)(N-2)(p-p_c(\sigma))}{(p-m)^2}\right]^{1/(p-m)}.
\end{equation*}
Actually, this stationary solution exists for $p>p_c(\sigma)$.

\subsection{Critical points at infinity}

To study the critical points of the system \eqref{PSsyst} located at infinity, we employ the compactification known as the Poincar\'e hypersphere. Following the theory presented in \cite[Section 3.10]{Pe}, we introduce the projective coordinates
$$
X=\frac{\overline{X}}{W}, \qquad Y=\frac{\overline{Y}}{W}, \qquad Z=\frac{\overline{Z}}{W}.
$$
In these variables, the equilibria on the hypersphere at infinity correspond to the solutions of the system in \cite[Theorem 4, Section 3.10]{Pe}, which in our setting reads
\begin{equation*}
\left\{\begin{array}{ll}\frac{1}{\sigma+2}\overline{X}\overline{Y}[(p-m)\overline{X}+(\sigma+2)\overline{Y}]=0,\\
(p-1)\overline{X}\overline{Z}\overline{Y}=0,\\
\frac{1}{\sigma+2}\overline{Y}\overline{Z}[p(\sigma+2)\overline{Y}+(p-m)\overline{X}]=0,\end{array}\right.
\end{equation*}
together with the requirement of lying on the equator of the Poincar\'e hypersphere. The latter requirement imposes the constraint $W=0$ and the normalization condition
$$
\overline{X}^2+\overline{Y}^2+\overline{Z}^2=1.
$$
Taking into account that $\overline{X}\geq0$ and $\overline{Z}\geq0$, the solutions of the above system yield the following set of critical points at infinity:
\begin{equation*}
\begin{split}
&Q_1=(1,0,0,0), \ \ Q_{2,3}=(0,\pm1,0,0), \ \ Q_4=(0,0,1,0), \ \ Q_{\gamma}=\left(\gamma,0,\sqrt{1-\gamma^2},0\right),\\
&Q_5=\left(\frac{\sigma+2}{\sqrt{(\sigma+2)^2+(p-m)^2}},-\frac{p-m}{\sqrt{(\sigma+2)^2+(p-m)^2}},0,0\right),
\end{split}
\end{equation*}
with $\gamma\in(0,1)$. We now turn to the study of these critical points. We give only a brief overview, and refer the reader to \cite{IMS23b, IS26} for the corresponding detailed computations. For the points $Q_1$, $Q_5$ and $Q_{\gamma}$, we work with the system obtained after projecting onto the $X$ variable, following \cite[Theorem 5(a), Section 3.10]{Pe}. This yields
\begin{equation}\label{PSinf1}
\left\{\begin{array}{ll}\dot{x}=x[(m-1)y-2x],\\
\dot{y}=-y^2-\frac{p-m}{\sigma+2}y-x-Nxy-xz,\\
\dot{z}=z[(p-1)y+\sigma x],\end{array}\right.
\end{equation}
with the lowercase variables defined by
\begin{equation}\label{change2}
x=\frac{1}{X}, \qquad y=\frac{Y}{X}, \qquad z=\frac{Z}{X},
\end{equation}
and the independent variable in \eqref{PSinf1} satisfies the differential equation
\begin{equation}\label{changexi}
\frac{d\eta_1}{d\xi}=\frac{\alpha}{m}\xi f^{1-m}(\xi)=\dfrac{X(\xi)}{\xi}.
\end{equation}
In these coordinates, the equilibria corresponding to $Q_1$, $Q_5$ and $Q_{\gamma}$ become
$$
Q_1=(0,0,0), \ \ Q_5=\left(0,-\frac{p-m}{\sigma+2},0\right), \ \ Q_{\gamma}=(0,0,\kappa), \ \ {\rm where} \ \kappa=\frac{\sqrt{1-\gamma^2}}{\gamma}\in(0,\infty).
$$
The study of the local behavior for the critical points $Q_1$ and $Q_\gamma$ of system \eqref{PSinf1} proceeds along the same lines as in \cite[Lemma 2.5]{IS26} (and its proof), which we summarize in the following technical result.
\begin{lemma}\label{lem.Q1Qg}
\begin{enumerate}
  \item The critical point $Q_1$ has two-dimensional center manifolds with stable flow together with a one-dimensional stable manifold. It thus behaves as a stable node for trajectories coming from the finite part of the phase space associated to the system \eqref{PSsyst}. The corresponding profiles satisfies the asymptotic behavior \eqref{beh.Q1f} as $\xi\to\infty$.
  \item For every $\gamma\in(0,1)$, the critical point $Q_{\gamma}$ only admits trajectories contained entirely in the invariant plane $\{x=0\}$ of the system \eqref{PSinf1}.
\end{enumerate}
\end{lemma}
\begin{proof}[Proof of the statement related to $Q_1$]
Due to its importance in the analysis and for the sake of completeness, we give a sketch of the proof of the statement related to the critical point $Q_1$. The linearization of the system \eqref{PSinf1} in a neighborhood of $Q_1$ has the matrix
$$M(Q_1)=\left(
         \begin{array}{ccc}
           0 & 0 & 0 \\[1mm]
           1 & -\frac{p-m}{\sigma+2} & 0 \\[1mm]
           0 & 0 & 0 \\
         \end{array}
       \right).
$$
By introducing the change of variable
$$
w=\frac{p-m}{\sigma+2}y+x,
$$
in order to put the system \eqref{PSinf1} in the canonical form for an application of the center manifold theory according to \cite{Carr} (see also \cite[Section 2.12]{Pe}) and by applying the center manifold theorem \cite[Theorem 3, Section 2.5]{Carr} (see some more details in the proof of \cite[Lemma 2.1]{IMS23}), we derive the equation of the center manifold of $Q_1$ as
\begin{equation*}
w=-\frac{(\sigma+2)(N-2)(p_c(\sigma)-p)}{(p-m)^2}x^2-xz+o(|(x,z)|^2).
\end{equation*}
We next replace
$$
y=\frac{\sigma+2}{p-m}(w-x)
$$
in the first and third equation of the system \eqref{PSinf1} according to \cite[Theorem 2, Section 2.4]{Carr} to deduce the direction of the flow. More precisely, the reduced system on the center manifold is given by
\begin{equation*}
\left\{\begin{array}{ll}\dot{x}&=-\frac{1}{\beta}x^2+x^2O(|(x,z)|),\\[1mm]
\dot{z}&=-\frac{1}{\beta}xz+xO(|(x,z)|^2),\end{array}\right.
\end{equation*}
and the stability of the flow on any center manifold in a neighborhood of $Q_1$ is obvious due to the positivity of $x$, $z$ and $\beta$. Since $p>m$, the only nonzero eigenvalue is also negative and thus the critical point $Q_1$ behaves like an attractor, as claimed.
\end{proof}
An important difference with respect to the previously mentioned references appears in the analysis of the critical point $Q_5$, which is essential in the proof of our main results. Let us recall here that we are working with $m>1$.
\begin{lemma}\label{lem.Q5}
The critical point $Q_5$ is a saddle point, with a two-dimensional stable manifold and a one-dimensional unstable manifold, the latter being contained in the invariant $y$-axis. The trajectories approaching $Q_5$ display the interface asymptotic behavior described in \eqref{beh.Q5f} and, more precisely,
\begin{equation}\label{beh.Q5m}
		f(\xi)\sim\left[\frac{(m-1)(p-m)\alpha}{m(\sigma+2)}\xi_0\right]^{1/(m-1)}(\xi_0-\xi)^{1/(m-1)}, \quad {\rm as} \ \xi\to\xi_0, \quad \xi<\xi_0.
\end{equation}
\end{lemma}
\begin{proof} The linearization of system \eqref{PSinf1} in a neighborhood of the critical point $Q_5$ is given by the matrix
$$
M(Q_5)=\left(
         \begin{array}{ccc}
           \frac{-(m-1)(p-m)}{\sigma+2} & 0 & 0 \\[1mm]
           -1+\frac{N(p-m)}{\sigma+2} & \frac{p-m}{\sigma+2} & 0 \\[1mm]
           0 & 0 & -\frac{(p-m)(p-1)}{\sigma+2} \\
         \end{array}
       \right),
$$
with eigenvalues satisfying, under our range of parameters, that
$$
\lambda_1=-\dfrac{(m-1)(p-m)}{\sigma+2}<0,\qquad\lambda_2=\dfrac{p-m}{\sigma+2}>0,\qquad\lambda_3=-\dfrac{(p-1)(p-m)}{\sigma+2}<0,
$$
and the corresponding eigenvectors
$$
e_1=\left(1,\dfrac{\sigma+2-N(p-m)}{m(p-m)},0\right),\qquad e_2=(0,1,0), \qquad e_3=(0,0,1).
$$
Therefore, from the uniqueness of an unstable manifold \cite[Theorem 3.2.1]{GH} and the invariance of the $y$-axis for the system \eqref{PSinf1}, we obtain that the unstable manifold is entirely contained in the $y$-axis. On the other hand, on the two-dimensional stable manifold, the trajectories enter to $Q_5$ tangent to the plane spanned by $e_1$ and $e_3$.

The local behavior in a neighbourhood of the critical point $Q_5$ follows from the fact that $y(\eta_1)\to-(p-m)/(\sigma+2)$ as $\eta_1\to\infty$. Replacing this limit into the first equation of \eqref{PSinf1}, we obtain, in a first order approximation, that
$$
\frac{dx}{d\eta_1}=-\frac{(p-m)(m-1)}{\sigma+2}x(\eta_1)+o(x(\eta_1)),
$$
which, with respect to the independent variable $\eta=\ln\,\xi$, writes
$$
\frac{dx}{d\eta}=\frac{1}{x}\frac{dx}{d\eta_1}\sim-\frac{(p-m)(m-1)}{\sigma+2}.
$$
This first order approximation readily implies that $\eta\to\eta_0\in\real$ on trajectories approaching $Q_5$. Thus, the fact that $\eta_1\to\infty$ along orbits entering $Q_5$ through its stable manifold translates, at the level of the profile, into a limit $\xi\to\xi_0\in(0,\infty)$ from the left. From this, and by reverting the change of variables \eqref{change2}, we obtain
$$
Y\sim-\dfrac{p-m}{\sigma+2}X,\qquad{\rm as}\ \xi\to\xi_0\in(0,\infty),\quad \xi<\xi_0,
$$
which, by undoing the change of variables \eqref{PSchange}, becomes after easy algebraic manipulations
$$
\lim\limits_{\xi\to\xi_0^{-}}(f^{m-1})'(\xi)=-\frac{(m-1)(p-m)\alpha}{m(\sigma+2)}\xi_0.
$$
We finally infer from the latter limit by an application of L'Hopital's rule that
$$
\lim\limits_{\xi\to\xi_0^{-}}\frac{f^{m-1}(\xi)}{\xi_0-\xi}=\frac{(m-1)(p-m)\alpha}{m(\sigma+2)}\xi_0,
$$
leading to the claimed interface behavior \eqref{beh.Q5f}. Actually, one directly obtains the more precise local approximation \eqref{beh.Q5m}.
\end{proof}
To analyse the critical points $Q_2$ and $Q_3$, we rely on the projected system in the $Y$ variable, obtained according to \cite[Theorem 5(b), Section 3.10]{Pe}, namely
\begin{equation}\label{PSinf2}
\left\{\begin{array}{lll}\pm\dot{x}=-x-Nxw-\frac{p-m}{\sigma+2}x^2-x^2w-xzw,\\[1mm]
\pm\dot{z}=-pz-(N+\sigma)zw-\frac{p-m}{\sigma+2}xz-xzw-z^2w,\\[1mm]
\pm\dot{w}=-mw-(N-2)w^2-\frac{p-m}{\sigma+2}xw-xw^2-zw^2,\end{array}\right.
\end{equation}
where the new variables $x$, $z$, $w$ are obtained from those of the system \eqref{PSsyst}, through the projection
\begin{equation*}
x=\frac{X}{Y}, \qquad z=\frac{Z}{Y}, \qquad w=\frac{1}{Y}.
\end{equation*}
Thus, in a neighborhood of the critical points $Q_2$ and $Q_3$, the flow of the system \eqref{PSsyst} becomes topologically equivalent to that of system \eqref{PSinf2} near the origin, with the minus corresponding to $Q_2$ and the plus sign to $Q_3$. The study of the dynamics at these points is relatively direct and follows the same approach as in \cite[Lemma 3.3]{IS25b}, to which we refer. This yields the same behavior as described in \cite[Lemma 2.6]{IS26}.
\begin{lemma}\label{lem.Q23}
The critical points $Q_2$ and $Q_3$ are, respectively, an unstable node and a stable node. The trajectories emerging from $Q_2$ correspond to profiles $f(\xi)$ for which there exists some $\xi_0\in(0,\infty)$ and $\delta>0$ such that
\begin{equation*}
f(\xi_0)=0, \qquad (f^m)'(\xi_0)=C>0, \qquad f>0 \ {\rm on} \ (\xi_0,\xi_0+\delta),
\end{equation*}
Similarly, the trajectories entering the stable node $Q_3$ correspond to profiles $f(\xi)$ for which there exists some $\xi_0\in(0,\infty)$ and $\delta\in(0,\xi_0)$ such that
\begin{equation*}
f(\xi_0)=0, \qquad (f^m)'(\xi_0)=-C<0, \qquad f>0 \ {\rm on} \ (\xi_0-\delta,\xi_0).
\end{equation*}
\end{lemma}
We are therefore left with the critical point $Q_4$, whose local study is considerably more delicate due to its non-hyperbolic nature. Nevertheless, by applying the same approach as in \cite[Section 4 and Appendix]{IMS23b}, one can show that no relevant trajectories enter or leave $Q_4$ from the finite part of the phase space. We recall that, from the location of the point on the Poincar\'e hypersphere, on the trajectories connecting to $Q_4$ we have the following limits
\begin{equation}\label{interm19}
Z(\eta)\to\infty, \quad \frac{Z(\eta)}{X(\eta)}\to\infty, \quad \frac{Z(\eta)}{Y(\eta)}\to\pm\infty.
\end{equation}
We add up for the moment the condition
\begin{equation}\label{interm20}
\frac{Z(\eta)}{X^2(\eta)}\to\infty, \quad {\rm or, equivalently,} \quad x(\eta_1)z(\eta_1)\to\infty,
\end{equation}
all the previous limits being taken as $\eta\to\eta^+$, where $\eta^+$ (which can be either finite or infinite) is the maximal interval of definition of the trajectory. The following technical result, showing that \eqref{interm19} and \eqref{interm20} (expressed in terms of profiles after undoing the change of variable \eqref{PSchange}) are incompatible with the equation \eqref{ODE.forward}, can be reconstructed from the arguments provided in \cite[Appendix]{IMS23b}.
\begin{lemma}\label{lem.Q4}
There are no solutions $f$ to Eq. \eqref{ODE.forward} satisfying simultaneously the following conditions
\begin{equation*}
\begin{split}
&\xi^{\sigma}f(\xi)^{p-1}\to\infty, \qquad \xi^{\sigma+2}f(\xi)^{p-m}\to\infty, \qquad \xi^{\sigma-2}f(\xi)^{m+p-2}\to\infty, \\
&\xi^{\sigma+1}f(\xi)^{p-m+1}(f')^{-1}(\xi)\to\pm\infty,
\end{split}
\end{equation*}
where the limits may occur in any of the possible cases $\xi\to0$, $\xi\to\xi_0\in(0,\infty)$ or $\xi\to\infty$.
\end{lemma}
On the other hand, introducing the change of variables $w=xz$ in \eqref{PSinf1}, we obtain the dynamical system
\begin{equation}\label{PSsyst3.ext}
\left\{\begin{array}{lll}\dot{x}=x[(m-1)y-2x],\\[1mm]
\dot{y}=-y^2-\frac{p-m}{\sigma+2}y-x-Nxy-w,\\[1mm]
\dot{w}=w[(\sigma-2)x+(m+p-2)y],\end{array}\right.
\end{equation}
According to the Lemma \ref{lem.Q4}, any trajectory that could potentially connect to $Q_4$ must satisfy \eqref{interm19} and thus \eqref{interm20} should not hold true, that is, $w(\eta_1)<\infty$. Consequently, such trajectories can approach only those finite critical points of system \eqref{PSsyst3.ext} for which $w<\infty$. The system possesses exactly two critical points with this property,
$$
Q_1'=(0,0,0),\qquad Q_5'=(0,-\dfrac{p-m}{\sigma+2},0).
$$
The analysis of their local behavior proceeds in exactly the same manner as in Lemmas \ref{lem.Q1Qg} and \ref{lem.Q5} (we leave the straightforward details to the reader). We thus have:
\begin{lemma}\label{lem.Q4.bis}
All the trajectories of system \eqref{PSsyst3.ext} that connect to the critical points $Q_1'$ and $Q_5'$ are the same ones inherited from the corresponding critical points $Q_1$ and $Q_5$.
\end{lemma}
Thus, these two critical points will not play any specific role in the analysis that follows.

\subsection{Dimensions $N=1$ and $N=2$}

Since all the previous analysis has been carried out in dimension $N\geq3$, it remains to address the low-dimensional cases $N=1$ and $N=2$, when $p_c(\sigma)=p_S(\sigma)=+\infty$ and we are always in the regime corresponding to $p<p_S(\sigma)$ considered in the higher-dimensional analysis. Thus, the critical point $P_2$ is absent and the local structure of the system \eqref{PSsyst} near the remaining critical points $P_0$ and $P_1$ exhibits some technical differences with respect to the case $N\geq 3$. However, the qualitative analysis of the phase space in neighborhoods of these points follows the same scheme developed in \cite{IS26} and is collected in the following result.
\begin{lemma}\label{lem.N2N1}
(a) In dimension $N=2$, the critical points $P_0$ and $P_1$ coincide. Denoting by $P_0$ the resulting point, it is a saddle-node with a leading three-dimensional center-unstable manifold tangent to the $Y$-axis and a non-leading two-dimensional unstable manifold. The orbits in the center-unstable manifold tangent to the $Y$-axis give rise to profiles exhibiting a vertical asymptote of the form
	\begin{equation}\label{beh.P0.N2}
		f(\xi)=D(-\ln\,\xi)^{1/m}, \qquad {\rm as} \ \xi\to0, \qquad D>0,
	\end{equation}
On the other hand, the orbits lying in the unstable manifold correspond to profiles displaying one of the local behaviors \eqref{beh.P0f}, depending on the sign of $\sigma$.

\medskip

(b) In dimension $N=1$, the critical point $P_0$ is a stable node, while $P_1$ is a saddle point with a two-dimensional unstable manifold and a one-dimensional stable manifold contained in the $Y$-axis. The profiles associated with the orbits emanating from $P_0$ exhibit either the local behavior
	\begin{equation}\label{beh.P0.N1}
		f(\xi)\sim\left[A-K\xi\right]^{-2/(1-m)}, \qquad {\rm as} \ \xi\to0,
	\end{equation}
where $A>0$ and $K\in\real\setminus\{0\}$ are arbitrary constants, or one of the local behaviors \eqref{beh.P0f}, according to the sign of $\sigma$.
\end{lemma}
A detailed proof of Lemma \ref{lem.N2N1} can be found in \cite[Section 7]{IS26}. Nevertheless, let us stress here that the orbits of interest, which give rise to the relevant local behaviors as $\xi\to0$, are tangent to the eigenspace associated with the eigenvalues $\lambda_1=2$ and $\lambda_3=\sigma+2$ of the linearization of the system \eqref{PSsyst} at $P_0$ and this is not changed in any way in dimensions $N=1$ and $N=2$.

\subsection{Phase space analysis for $m=1$}

Fixing throughout this section $m=1$, we easily observe that most of the local analysis of the system \eqref{PSsyst} becomes just a particular case of the local analysis performed in the previous section. However, there are two differences, one of them very significant:

$\bullet$ in the analysis of the critical point $P_0$, the local expansion given in \eqref{beh.P0f} is different for $m=1$ and $\sigma>0$. Indeed, we infer from \eqref{interm2} that, for $\sigma>0$, $Z(\eta)$ becomes of lower order than $X(\eta)$ in \eqref{lC}. Thus, we deduce that
$$
\lim\limits_{\eta\to-\infty}\frac{Y(\eta)}{X(\eta)}=-\frac{1}{N},
$$
which, after undoing the change of variable \eqref{PSchange} and taking into account that $X(\xi)=\alpha\xi^2$, becomes
\begin{equation}\label{interm4}
\lim\limits_{\xi\to0}\frac{f'(\xi)}{\xi f(\xi)}=-\frac{\alpha}{N}.
\end{equation}
An application of L'Hopital's rule in \eqref{interm4} gives
$$
\lim\limits_{\xi\to0}\frac{\ln f(\xi)-\ln f(0)}{\xi^2}=-\frac{\alpha}{2N},
$$
which readily gives the local expansion \eqref{beh.P0f} for $m>1$ and $\sigma>0$.

$\bullet$ The local analysis of the critical point $Q_5$ is different, and we present it below.
\begin{lemma}\label{lem.Q5m1}
For $m=1$, the critical point $Q_5$ has a one-dimensional stable manifold, a one-dimensional unstable manifold and one-dimensional center manifolds with stable direction of the flow. The trajectories approaching $Q_5$ on the two-dimensional center-stable manifold correspond to profiles with a Gaussian-like tail at infinity as given in \eqref{beh.Q5heat}.
\end{lemma}
\begin{proof}
The linearization of the system \eqref{PSinf1} in a neighborhood of $Q_5$ has the matrix
$$
M(Q_5)=\left(
         \begin{array}{ccc}
           0 & 0 & 0 \\[1mm]
           -1+\frac{N(p-1)}{\sigma+2} & \frac{p-1}{\sigma+2} & 0 \\[1mm]
           0 & 0 & -\frac{(p-1)^2}{\sigma+2} \\
         \end{array}
       \right).
$$
According to \cite[Theorem 3.2.1]{GH}, the center manifolds are all tangent to the eigenvector corresponding to the zero eigenvalue of the matrix $M(Q_5)$, which is $e_1=(1,-N+(\sigma+2)/(p-1),0)$, while the direction of the flow on the center manifolds is given by $\dot{x}=-2x^2<0$, that is, a stable direction. The center manifolds couple with the unique stable manifold to form a two-dimensional center-stable manifold. We deduce from the first equation $\dot{x}=-2x^2$ of the system \eqref{PSinf1} that, with respect to the initial independent variable $\eta=\ln\,\xi$, we have $x'(\eta)=-2x(\eta)$, that is, $x(\eta)=e^{-2\eta}$. Since $x\to0$ when approaching $Q_5$, it follows that $\eta\to\infty$. Thus, for $m=1$, the convergence $\eta_1\to\infty$ along orbits entering $Q_5$ through its center-stable manifold translates, at the level of the profile, into a behavior as $\xi\to\infty$. In order to establish the precise behavior as $\xi\to\infty$, we start from
$$
\lim\limits_{\xi\to\infty}\frac{Y(\xi)}{X(\xi)}=\lim\limits_{\eta_1\to\infty}y(\eta_1)=-\frac{p-1}{\sigma+2},
$$
which, after undoing the change of variable \eqref{PSchange}, leads to
\begin{equation}\label{interm1}
\lim\limits_{\xi\to\infty}\frac{f'(\xi)}{\xi f(\xi)}=-\frac{1}{2}.
\end{equation}
An application of L'Hopital's rule together with \eqref{interm1} ensure that
$$
\lim\limits_{\xi\to\infty}\frac{\ln f(\xi)}{\xi^2}=-\frac{1}{4},
$$
and readily gives the Gaussian-like behavior \eqref{beh.Q5heat}, completing the proof.
\end{proof}

\section{Proof of the main results}\label{sec.global.super}

In this section we give the proof of Theorems \ref{th.global.heat} and \ref{th.global.super}. Let us remark that, in view of the local analysis performed in the previous Section, we are looking for trajectories in the phase space associated to the system \eqref{PSsyst} connecting $P_0$ to either $Q_1$ or $Q_5$. Since, at the level of the phase space, the case $m=1$ is just a particular case of the general range $m\geq1$, the same proof is valid for both theorems.

The general approach for the proof is a shooting method with respect to the parameter $C\in(0,\infty)$ on the trajectories $l_C$ defined in \eqref{lC} composing the unstable manifold of $P_0$. We thus begin with two preparatory results investigating the behavior of the two limiting trajectories, that is $l_0$ (contained in the invariant plane $\{Z=0\}$) and $l_{\infty}$ (contained in the invariant plane $\{X=0\}$). Regarding this latter trajectory, we have the following result:
\begin{proposition}\label{prop.X0}
Let $N\geq1$, and $m$, $p$, $\sigma$ as in \eqref{range.exp}. Then the orbit $l_{\infty}$, stemming from $P_0$ inside the invariant plane $\{X=0\}$:

$\bullet$ connects to the stable node at infinity $Q_3$ whenever $p<p_S(\sigma)$.

$\bullet$ connects to the critical point $P_1$ when $p=p_S(\sigma)$ and has the explicit expression
\begin{equation}\label{cylinder}
Z=-\frac{N+\sigma}{N-2}(mY+N-2)Y.
\end{equation}

$\bullet$ connects to the critical point $P_2$ for $p>p_S(\sigma)$.
\end{proposition}
The invariant plane $\{X=0\}$ of the system \eqref{PSsyst} is completely identical to the one studied in \cite[Section 5]{IMS23b}, and we refer the reader to the quoted reference for a proof.

We now turn our attention to the analysis of the remaining limit trajectory, namely $l_0$, which corresponds to $C=0$ in \eqref{lC} and is therefore entirely contained in the invariant plane $\{Z=0\}$. At this stage, the Fujita exponent enters the picture and plays a decisive role in the behavior of the solutions.
\begin{proposition}\label{prop.Z0}
Let $N\geq1$, and $m$, $p$, $\sigma$ as in \eqref{range.exp}. Then the trajectory $l_0$ connects to the saddle point $Q_5$ when $p=p_F(\sigma)$ and to the non-hyperbolic point $Q_1$ which exhibits the behavior of a stable node when $p>p_F(\sigma)$.
\end{proposition}
\begin{proof}
In the plane $\{Z=0\}$, the system \eqref{PSsyst} reduces to
\begin{equation}\label{systZ0}
\left\{\begin{array}{ll}\dot{X}=X(2-(m-1)Y),\\\dot{Y}=-X-(N-2)Y-mY^2-\frac{p-m}{\sigma+2}XY,\end{array}\right.
\end{equation}
or, equivalently, in variables $(x,y)$ given by \eqref{change2},
\begin{equation}\label{PSinf1z0}
	\left\{\begin{array}{ll}\dot{x}=x[(m-1)y-2x],\\
		\dot{y}=-y^2-\frac{p-m}{\sigma+2}y-x-Nxy.\\
		\end{array}\right.
\end{equation}
Consider the line $\{Y=-X/N\}$, which, except for the initial point $P_0$, is equivalent to $y=-1/N$ in the system \eqref{PSinf1z0}. A direct computation shows that this line is a solution to \eqref{PSinf1z0} if and only if $p=p_F(\sigma)$. Indeed, plugging $y(\eta_1)=-1/N$ in the second equation of \eqref{PSinf1z0}, we readily get that
$$
\frac{p-m}{N(\sigma+2)}=\frac{1}{N^2},
$$
which implies $p=p_F(\sigma)$. Moreover, if $p=p_F(\sigma)$ it is obvious that the trajectory $Y=-X/N$ (or $y=-1/N$) connects $P_0$ to $Q_5$, as claimed.

Consider next the normal vector $\overline{n}=(1/N,1)$ associated to the line $\{Y=-X/N\}$. The direction of the flow of the system \eqref{systZ0} across this line is given by the sign of the scalar product between $\overline{n}$ and the vector field of the system, which yields
\begin{equation}\label{interm7}
	F(X)=-\frac{(mN-Np+\sigma+2)X^2}{N^2(\sigma+2)}=\frac{p-p_F(\sigma)}{N(\sigma+2)}X^2.
\end{equation}
To determine on which side of the line $\{Y=-X/N\}$ the orbit $l_0$ emerges from $P_0$, we compute its second order expansion near $(X,Y)=(0,0)$. We proceed as in \cite[Section 2.7]{Shilnikov} to identify the Taylor expansion of the unstable manifold. We set
\begin{equation}\label{interm8}
Y=\psi(X):=-\frac{X}{N}+KX^2+o(X^2), \qquad K\in\real,
\end{equation}
then insert \eqref{interm8} into \eqref{systZ0} and compute the expansion of $\psi$. Matching coefficients gives
\begin{equation}\label{interm9}
Y(\eta)=-\frac{1}{N}X(\eta)+\frac{p-p_F(\sigma)}{N(N+2)(\sigma+2)}X^2(\eta)+o(X^2(\eta)), \qquad {\rm as} \ \eta\to-\infty.
\end{equation}
Moreover, it is obvious that the flow of the system \eqref{systZ0} across the line $\{Y=0\}$ points into the negative half-plane. We infer from the expansion \eqref{interm9} and the sign of $F(X)$ in \eqref{interm7} that, for $p>p_F(\sigma)$, the trajectory $l_0$ enters and remains in the region $\{X>0,-X/N<Y<0\}$, corresponding in the $(x,y)$-variables to the region $\{x>0,-1/N<y<0\}$, while $y(Q_5)<-1/N$, so again $Q_5$ belongs to the opposite region and cannot be the endpoint of the trajectory $l_0$. Moreover, in the $(x,y)$-variables, the trajectory remains in a compact set since $x$ is decreasing with respect to $\eta_1$, while $-1/N<y(\eta_1)<0$ for any $\eta_1\in\real$. By the Poincar\'e-Bendixon's Theorem, the trajectory $l_0$ must reach a critical point in the closed strip $\{-1/N\leq y\leq 0\}$, and the only such point is $Q_1$, completing the proof.
\end{proof}
The results derived so far provide all the ingredients needed to prove Theorems \ref{th.global.super} and \ref{th.global.heat}, the latter proof being actually (at the level of the phase space) only a particular case of the former. We recall that some parts of the argument are inspired by the proof of \cite[Theorem 1.1]{IS26}.
\begin{proof}[Proof of Theorem \ref{th.global.super}]
Fix $N\geq1$ and $m$, $p$, $\sigma$ as in \eqref{range.exp}. In view of the previous analysis of the trajectories, we distinguish the following two cases in the proof.

\medskip

\textbf{(a) Range $p_F(\sigma)<p<p_S(\sigma)$.} To analyze this range, we introduce the following sets:
\begin{equation*}
\begin{split}
&\mathcal{A}=\{C\in(0,\infty): {\rm the \ trajectory} \ l_C \ {\rm reaches \ the \ point} \ Q_3\},\\
&\mathcal{B}=\{C\in(0,\infty): {\rm the \ trajectory} \ l_C \ {\rm reaches \ neither } \ Q_3 \ {\rm nor} \ Q_1\},\\
&\mathcal{C}=\{C\in(0,\infty): {\rm the \ trajectory} \ l_C \ {\rm reaches \ the \ point} \ Q_1\}.
\end{split}
\end{equation*}
Since $Q_3$ and $Q_1$ are attractors for the system \eqref{PSsyst} as shown in Lemmas \ref{lem.Q23} and \ref{lem.Q1Qg}, it follows that the sets $\mathcal{A}$ and $\mathcal{C}$ are open. On the one hand, Proposition \ref{prop.X0} gives that the trajectory $l_{\infty}$ converges to $Q_3$ and by the stability of this critical point, the set $\mathcal{A}$ is nonempty and in fact contains an open interval of the form $(C^*,\infty)$. On the other hand, Proposition \ref{prop.Z0} ensures that the orbit $l_0$ reaches $Q_1$ and, by its stability and continuity arguments, one deduces that $\mathcal{C}$ is also nonempty and contains an interval $(0,C_*)$. From these facts and standard topology arguments, the set $\mathcal{B}$ is nonempty.

Fix now $C_0\in\mathcal{B}$. Our goal is to show that the corresponding trajectory $l_{C_0}$ reaches the saddle point $Q_5$. Note first that the direction of the flow of the system \eqref{PSsyst} across the plane $\{Y=0\}$ is given by the sign of $-X-Z\leq0$, hence the trajectories $l_C$ for any $C\in(0,\infty)$ remain forever in the region $\{X\geq0, Y<0, Z\geq0\}$. It follows then from the first equation of the system \eqref{PSsyst} that $X$ is monotonically increasing. Since the monotonicity properties of $Z$ are more delicate to analyze, we instead study the existence of its limit as $\xi\to\infty$. To this end, borrowing ideas from \cite{IS26}, we introduce the function
$$
g(\xi):=\xi^{(\sigma+2)/(p-m)}f(\xi).
$$
A straightforward computation based on \eqref{ODE.forward} shows that $g$ is a solution to the following differential equation:
\begin{equation}\label{interm11}
\begin{split}
\xi^2(g^m)''(\xi)&+\left(N-1-\frac{2m(\sigma+2)}{p-m}\right)\xi(g^m)'(\xi)+\frac{m(\sigma+2)(m\sigma+m+p)}{(p-m)^2}g^m(\xi)\\
&+\frac{p-m}{L}\xi^{(m-1)(\sigma+2)/(p-m)+3}g'(\xi)+g^p(\xi)=0.
\end{split}
\end{equation}
We observe that any solution $g(\xi)$ of \eqref{interm11} cannot admit local minima. Consequently, either $g$ is monotone or it possesses a single local maximum point and is eventually decreasing. Since
$$
g(\xi)=(mZ(\xi))^{1/(p-m)},
$$
it follows that there exists $\overline{\eta}$ such that $\eta\mapsto Z(\eta)$ is monotone for $\eta>\overline{\eta}$. Let $\eta^+$ be the upper edge of the maximum interval of definition of the trajectory $l_{C_0}$ (note that $\eta^+$ can be a priori either finite or infinite). Then, there exist (along the trajectory $l_{C_0}$) the limits
$$
X_{\infty}:=\lim\limits_{\eta\to\eta^+}X(\eta), \quad Z_{\infty}:=\lim\limits_{\eta\to\eta^+}Z(\eta).
$$
Let us prove first that $X_{\infty}=\infty$. Assume for contradiction that this is not the case, thus $X_{\infty}<\infty$. If also $Z_{\infty}<\infty$, the monotonicity (proved above) of both $X$ and $Z$ and an argument of local maxima and minima of $Y(\eta)$ similar to the one given at the end of this proof easily ensure that either the trajectory $l_{C_0}$ ends at a finite critical point (and there is none) or that there is a subsequence $\eta_k\to\eta^+$ such that $Y(\eta_k)\to-\infty$. But since $Q_3$ is a stable node, it then follows that the trajectory connects to $Q_3$ and thus $C_0\in\mathcal{A}$, which is a contradiction. We then deduce that $Z_{\infty}=\infty$, which at the same time implies that $Z$ is increasing for $\eta\in(\overline{\eta},\eta^+)$. The third equation of the system \eqref{PSsyst} yields thus that
\begin{equation}\label{interm12}
-\frac{\sigma+2}{p-m}<Y(\eta)<0, \quad \eta\in(\overline{\eta},\eta^+),
\end{equation}
whence the trajectory $l_{C_0}$ connects to the critical point $Q_4$, contradicting the outcome of Lemmas \ref{lem.Q4} and \ref{lem.Q4.bis}. These contradictions imply that $X_{\infty}=\infty$.

Suppose first that $Z_{\infty}=\infty$. Then, as explained above, \eqref{interm12} holds true, which implies that, in the $(x,y,z)$-coordinates defined by \eqref{change2}, the $\omega$-limit set of the trajectory $l_{C_0}$ consists only of points with coordinates $x=y=0$; that is, a subset of the critical line composed by the points $Q_{\gamma}$, with $\gamma>0$, together with $Q_1$ and $Q_4$. We prove next that such a configuration is impossible: indeed, Lemmas \ref{lem.Q4}, \ref{lem.Q4.bis} and \ref{lem.Q1Qg} show that no trajectory coming from the finite part of the phase space can approach any $Q_{\gamma}$ with $\gamma>0$ or $Q_4$, while $Q_1$ behaves like an attractor and in this case we would have $C_0\in\mathcal{C}$. However, there is still the theoretical possibility that the $\omega$-limit set is composed by a continuum of points $Q_{\gamma}$ with $\gamma>0$. We deduce from \cite[Lemma 1, Section 2.4]{Carr} that, in a neighborhood of every $Q_{\gamma}$, a trajectory lies exponentially close to its center manifold. However, analyzing $Q_{\gamma}=(0,0,\kappa)$, where $\kappa=\sqrt{1-\gamma^2}/\gamma\in(0,\infty)$, in the system \eqref{PSinf1}, we find (by a change of variable $z=\kappa+\overline{z}$ in the second equation of it) that the center manifold of $Q_{\gamma}$ writes
$$
y=-\frac{\sigma+2}{p-m}(1+\kappa)x+o(x).
$$
Since $\kappa>0$ and $x>0$, it follows that $y<-(\sigma+2)x/(p-m)$ in a sufficiently small neighborhood of $Q_{\gamma}$ and thus $Y<-(\sigma+2)/(p-m)$, contradicting \eqref{interm12} and thus completing the argument.

We have thus deduced that $Z_{\infty}<\infty$ on the trajectory $l_{C_0}$. In this case, since $X_{\infty}=\infty$, the $\omega$-limit of $l_{C_0}$ is composed by points with $x=z=0$ in the $(x,y,z)$-variables defined by \eqref{change2}. We also deduce that $y$ remains bounded, that is, there is $y_0\in(0,\infty)$ such that $-y_0<y(\eta_1)<0$, for any $\eta_1\in\real$. Indeed, on the contrary, there is a subsequence on which $Y/X\to-\infty$, but in such case the trajectory $l_{C_0}$ ends at $Q_3$, since $Q_3$ is a stable node. Assume next that $y(\eta_1)$ oscillates infinitely many times and consider infinite sequences of local minima and maxima $(\eta_{1,k})_{k\geq1}$, respectively $(\eta_{1}^k)_{k\geq1}$, of the function $y(\eta)$ along the trajectory $l_{C_0}$ such that $\eta_{1,k}\to\infty$, $\eta_{1}^{k}\to\infty$ as $k\to\infty$. By evaluating the second equation of the system \eqref{PSinf1} at $\eta_{1,k}$, respectively at $\eta_1^{k}$, and taking into account that $x(\eta_{1,k})$, $z(\eta_{1,k})$ converge to zero as $k\to\infty$ (and similarly for the sequence of maxima), we get that
$$
0=\lim\limits_{k\to\infty}\left[y^2(\eta_{1,k})+\frac{p-m}{\sigma+2}y(\eta_{1,k})\right]
=\lim\limits_{k\to\infty}\left[y^2(\eta_{1}^k)+\frac{p-m}{\sigma+2}y(\eta_{1}^k)\right].
$$
It follows that either at least one of the sequences $(y(\eta_{1,k}))$, $(y(\eta_{1}^k))$ tends to zero, or both of them converge to $-(p-m)/(\sigma+2)$. The former case leads to a contradiction, since $Q_1$ is an attractor and this would imply that $l_{C_0}$ connects to $Q_1$, hence $C_0\in\mathcal{C}$. We are thus left with the latter case, which directly proves that $l_{C_0}$ is a connection between $P_0$ and $Q_5$.

In Figure \ref{fig1} we plot the outcome of a numerical experiment showing the trajectories starting from $P_0$ and reaching one of the critical points $Q_3$, $Q_5$ or $Q_5$, together with their corresponding profiles.

\begin{figure}[ht!]
  \begin{center}
  \subfigure[Trajectories starting from $P_0$]{\includegraphics[width=7.5cm,height=6cm]{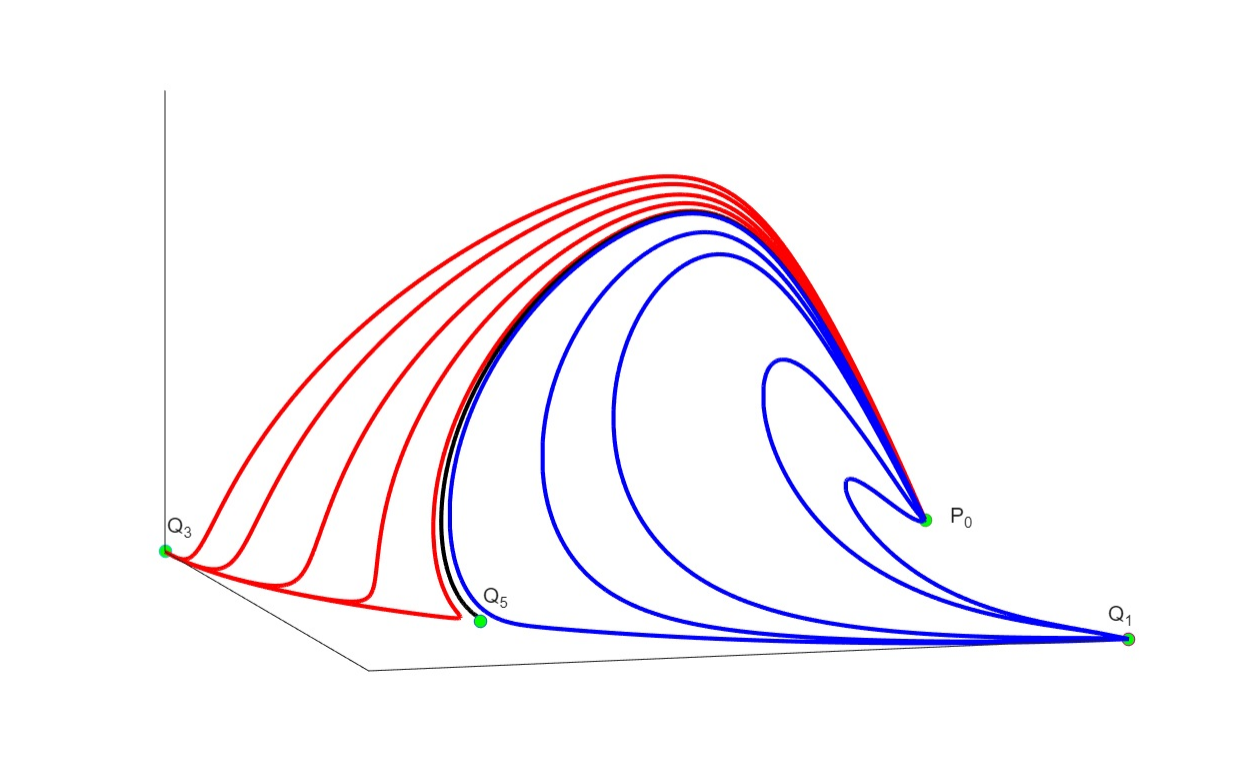}}
  \subfigure[Profiles corresponding to the previous trajectories]{\includegraphics[width=6cm,height=4.5cm]{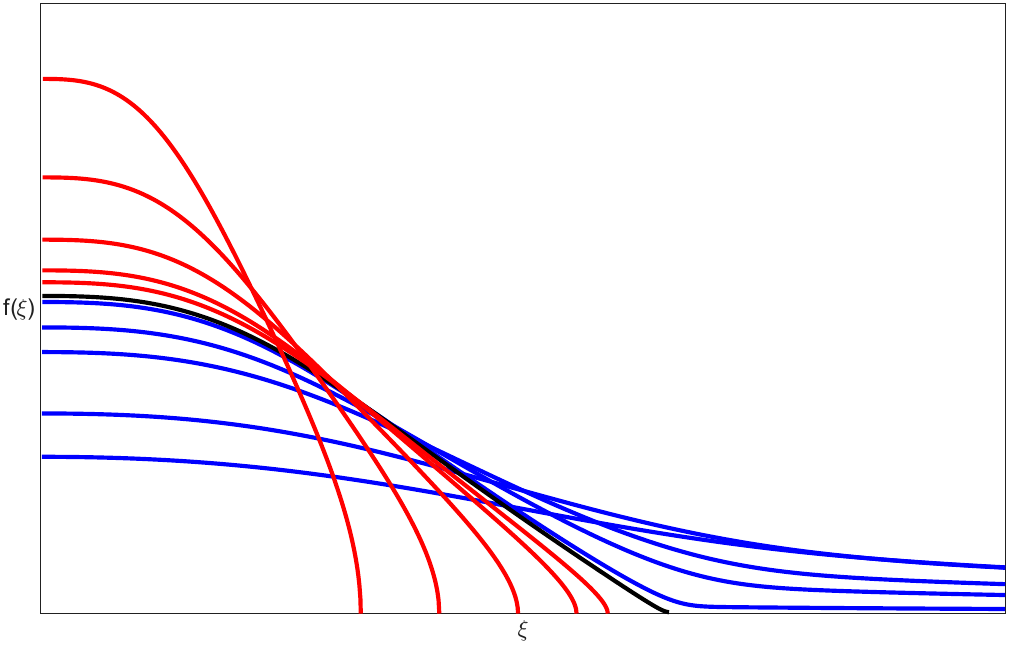}}
  \end{center}
  \caption{A shooting of various trajectories $l_C$ starting from $P_0$ and their corresponding profiles. Experiments for $m=2$, $N=3$, $p=5$ and $\sigma=1$.}\label{fig1}
\end{figure}

\medskip

\textbf{(b) Range $p_S(\sigma)\leq p<\infty$.} We recall first that this range is only possible if $N\geq3$. We work with the cylinder in $(X,Y,Z)$-coordinates whose base is the explicit curve \eqref{cylinder}. The direction of the flow of system \eqref{PSsyst} on the surface of this cylinder, with normal direction
$$
\overline{n}(X,Y,Z)=\left(0,\frac{N+\sigma}{N-2}(2mY+N-2),1\right)
$$
is determined by the sign of
\begin{equation}\label{flow.cyl.ext}
\begin{split}
E(X,Y;p)=\frac{N+\sigma}{N-2}&\left[(p_S(\sigma)-p)Y^2(mY+N-2)\right.\\&\left.-X\left(1+\frac{p-m}{\sigma+2}Y\right)(2mY+N-2)\right],
\end{split}
\end{equation}
which is clearly negative if $p=p_S(\sigma)$, since
$$
E(X,Y;p_S(\sigma))=-(N+\sigma)X\left(1+\frac{2m}{N-2}Y\right)^2<0.
$$
For $p>p_S(\sigma)$, the quantity $E(X,Y;p)$ may become positive only for values in the interval
\begin{equation}\label{interm10}
-\frac{N-2}{m}<Y<-\frac{\sigma+2}{p-m}.
\end{equation}
However, for trajectories located in the interior of the cylinder, we observe that $\dot{Z}(\eta)<0$ whenever $Y(\eta)$ satisfies \eqref{interm10}, while on the boundary of the cylinder the variable $Z$ still increases under the same condition. This prevents any orbit from going out of the interior of the cylinder once it enters it. It is immediate to check that, for every $p\geq p_S(\sigma)$, all trajectories $l_C$ associated to \eqref{lC} enter the interior of the cylinder, with the sole exception of $l_{\infty}$ when $p=p_S(\sigma)$ (which in that case coincides exactly with the base curve \eqref{cylinder} lying in the plane $\{X=0\}$). Consequently, every orbit $(l_C)$ with $C\in(0,\infty)$ remains confined inside the cylinder afterwards. Thus, $Z(\eta)$ and $Y(\eta)$ remain bounded for $\eta\in(-\infty,\eta^+)$ along any of the trajectories $l_C$. Since $X$ increases with $\eta$, by similar arguments as in Part (b) of this proof, we conclude that $X(\eta)\to\infty$ as $\eta\to\eta^+$ and thus, by definition, $l_C$ enters the attractor $Q_1$ for any $C\in(0,\infty)$, as claimed.

We plot in Figure \ref{fig2} the results of an experiment with trajectories starting from $P_0$ and connecting to $Q_1$ according to the previous proof.

\begin{figure}[ht!]
  \begin{center}
  \subfigure[Trajectories starting from $P_0$]{\includegraphics[width=7.5cm,height=6cm]{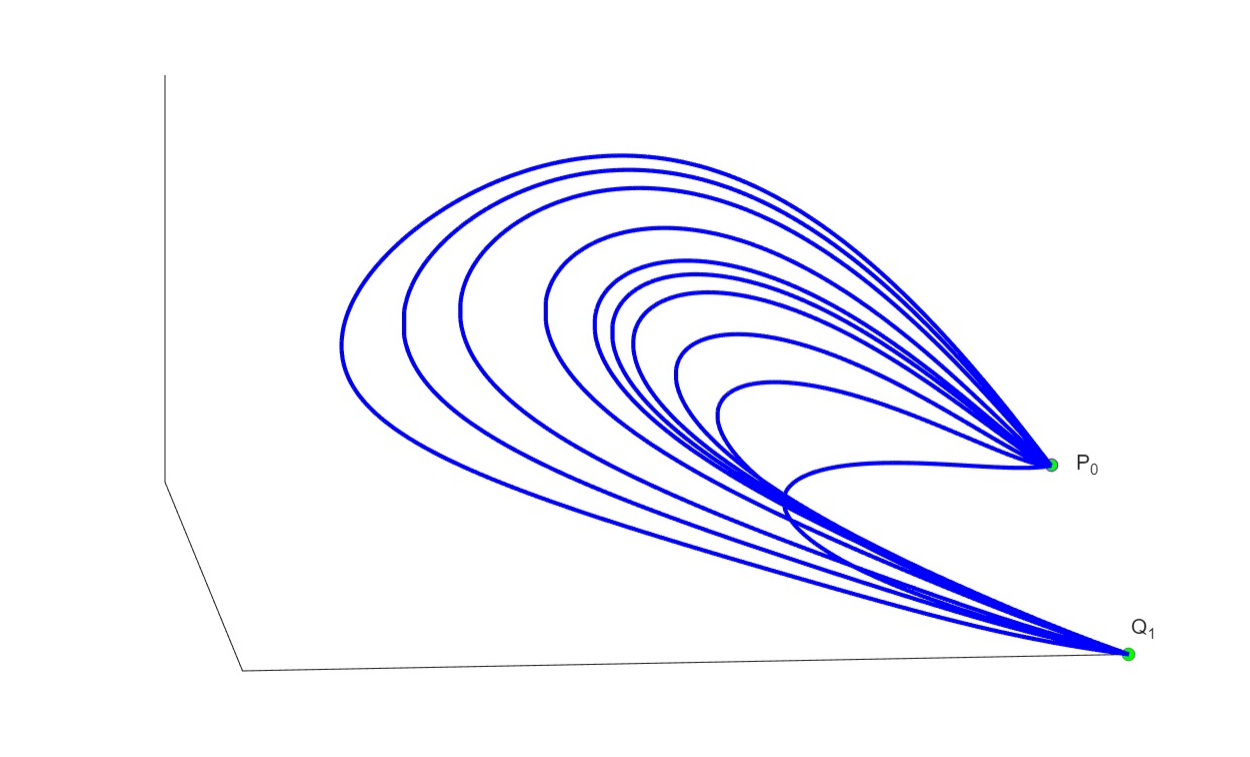}}
  \subfigure[Profiles corresponding to the previous trajectories]{\includegraphics[width=6cm,height=4.5cm]{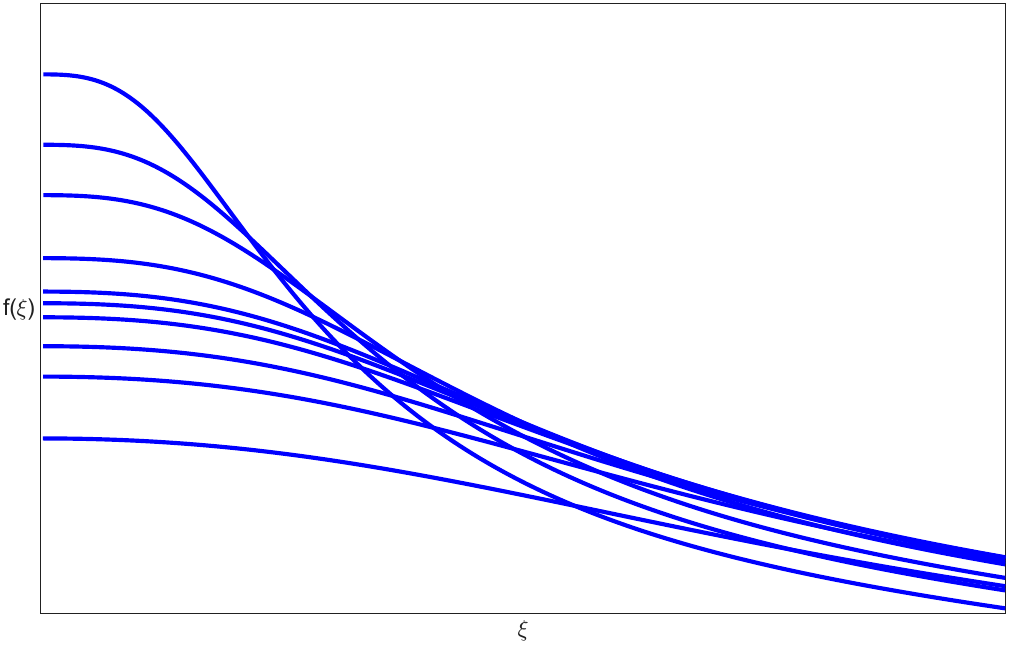}}
  \end{center}
  \caption{A shooting of various trajectories $l_C$ starting from $P_0$ and their corresponding profiles. Experiments for $m=2$, $N=3$, $p=15$ and $\sigma=1$.}\label{fig2}
\end{figure}

\medskip

The conclusion of both Theorems \ref{th.global.heat} and \ref{th.global.super} follows then from \eqref{bij} and the fact that, in the range $p_F(\sigma)<p<p_S(\sigma)$, the connections between $P_0$ and $Q_1$ correspond to orbits $l_C$ with $C\in(0,C_*)$, hence to profiles $f(\cdot;A)$ with $A\in(0,A_*)$, completing the proof.
\end{proof}

\bigskip

\noindent \textbf{Acknowledgements} The authors are partially supported by the Project PID2024-160967NB-I00 funded by AEI (Spain) and FEDER.

\bigskip

\noindent \textbf{Data availability} Our manuscript has no associated data.

\bigskip

\noindent \textbf{Conflict of interest} The authors declare that there is no conflict of interest.

\bibliographystyle{amsplain}

\end{document}